\documentclass[11pt,a4paper]{article}

\usepackage{amsmath,amsfonts,amssymb}
\usepackage{bbm}
\usepackage{cases}

\usepackage{subcaption}
\usepackage{kotex}

\usepackage{xcolor}
\usepackage{graphicx}
\usepackage{ulem}
\allowdisplaybreaks

\newtheorem{defn}{Definition}[section]
\newtheorem{theo}[defn]{Theorem}
\newtheorem{lem}[defn]{Lemma}

\newtheorem{exam}[defn]{Example}

\newenvironment{proof}{{\bf Proof }}{{\vskip 0.1cm \hfill$\Box$}}
\begin{document} 

\noindent
{\Large \bf Resolvent approaches to elliptic regularity in stationary Fokker-Planck equations}
\\ \\
\bigskip
\noindent
{\bf Haesung Lee}  \\
\noindent
{\bf Abstract.}  
This paper investigates the local regularity of solutions to stationary Fokker-Planck equations on an open set $U \subset \mathbb{R}^d$ with $d \geq 2$. A central objective is to relax the classical assumptions on the coefficients by focusing on the case where the drift vector field $\mathbf{G}$ is only assumed to be locally square-integrable, i.e. $\mathbf{G} \in L^2_{loc}(U, \mathbb{R}^d)$, the symmetric diffusion matrix $A = (a_{ij})_{1 \leq i,j \leq d}$ is assumed to be locally uniformly strictly elliptic and bounded, with coefficients satisfying $a_{ij} \in VMO_{loc}(U)$ for all $1 \leq i,j \leq d$ and ${\rm div}A \in L^2_{loc}(U, \mathbb{R}^d)$. Our main result shows that any locally bounded function $h \in L^\infty_{loc}(U)$ satisfying the stationary Fokker-Planck equation $L^*(h\, dx) = 0$ must in fact belong to the local Sobolev space $H^{1,2}_{loc}(U)$. The proof is based on the construction of a sub-Markovian resolvent associated with the principal elliptic operator $L^A := {\rm trace}(A \nabla^2)$, combined with delicate energy-type inequalities. In particular, we show that the density $h$ can be realized as the weak $H^{1,2}$-limit of images of resolvent operators. 


\section{Introduction} \label{intro}
The Fokker-Planck equation is fundamentally connected to stochastic analysis, particularly to the theory of stochastic differential equations (SDEs). Given a stochastic process $(X_t)_{t \geq 0}$ defined on a filtered probability space $(\Omega, \mathcal{F}, \mathbb{P}, (\mathcal{F}_t)_{t \geq 0})$ and a standard Brownian motion $(W_t)_{t \geq 0}$, consider the following SDE:
\[
X_t = x + \int_0^t \sigma(X_s) \, dW_s + \int_0^t \mathbf{G}(X_s) \, ds, \quad 0 \leq t < \infty, \quad \mathbb{P}\text{-a.s.}
\]
where $x \in \mathbb{R}^d$, $\sigma$ is a $d \times d$ matrix of functions, and $\mathbf{G}$ is a vector field that is locally integrable.
For any $f \in C_0^\infty(\mathbb{R}^d)$ and $t > 0$, applying It\^{o}'s formula (\cite[Chapter IV, Theorem 3.3]{R99}) yields
\begin{align*}
f(X_t) - f(X_0) &= \int_0^t \nabla f(X_s)  \, dX_s + \frac{1}{2} \sum_{i,j=1}^d \int_0^t \frac{\partial^2 f}{\partial x_i \partial x_j}(X_s) \, d\langle X^i, X^j \rangle_s \\
&= \int_0^t \nabla f(X_s) \sigma(X_s)\, dW_s + \int_0^t Lf(X_s) \, ds, \quad \text{ $\mathbb{P}$-a.s.},
\end{align*}
where the operator $L$ is given by
$$
Lf = {\rm trace}(A \nabla^2 f) + \langle \mathbf{G}, \nabla f \rangle, \quad \text{with } A = \frac12 \sigma \sigma^T.
$$
Taking expectations on both sides of the above identity, we obtain
\begin{equation}\label{fokkeriden}
\mathbb{E}[f(X_t)] - f(x) = \int_0^t \mathbb{E}[L f(X_s)]\,ds \quad \text{for all } f \in C_0^\infty(\mathbb{R}^d) \text{ and } t > 0,
\end{equation}
where $\mathbb{E}$ is the expectation with respect to $\mathbb{P}$. Let $p(t, x, \cdot)$ denote the probability density function of $X_t$ under the probability measure $\mathbb{P}$ (such a density exists under quite general assumptions and see \cite{L25ai}). Then, equation \eqref{fokkeriden} is equivalent to the following integral identity:
\begin{equation}\label{fokkerpllan}
\int_{\mathbb{R}^d} f(y)\, p(t, x, y)\,dy - f(x)
= \int_0^t \int_{\mathbb{R}^d} Lf(y)\, p(s, x, y)\,dy\,ds \quad \text{for all } f \in C_0^\infty(\mathbb{R}^d) \text{ and } t > 0.
\end{equation}
The family of measures $\left(p(t, x, y)\,dy\right)_{t \geq 0}$ is referred to as the solution to the Cauchy problem for the Fokker-Planck equation associated with the operator $L$ and the initial distribution $\delta_x$, the Dirac delta measure at $x$ (cf. \cite[Definition 6.1, Proposition 6.1.2]{BKRS15}).
We refer the reader to the introduction of \cite{L25ai} for a more detailed discussion of the Fokker-Planck equation. In fact, under appropriate ergodicity assumptions, one can show the existence of a probability measure $\nu$ on $\mathcal{B}(\mathbb{R}^d)$ such that
\begin{equation} \label{weakconver}
\lim_{t \to \infty} p(t, x, y)\,dy = \nu(dy) \quad \text{weakly on } \mathbb{R}^d,
\end{equation}
(see \cite[Theorem 1.1(iii)]{L25ai}).
Under suitable regularity assumptions on the coefficients $A$ and the drift vector field $\mathbf{G}$, it follows from \eqref{fokkerpllan} and \eqref{weakconver} that $L^* \nu=0$ on $\mathbb{R}^d$, i.e.
\begin{equation} \label{stafokkerp}
\int_{\mathbb{R}^d} Lf \, d\nu = 0  \quad \text{for all } f \in C_0^\infty(\mathbb{R}^d).
\end{equation}
A Borel measure $\nu$ satisfying \eqref{stafokkerp} is called a solution to the stationary Fokker-Planck equation associated with the operator $L$. In stochastic analysis, $\nu$ is also called an infinitesimally invariant measure for $(L, C_0^{\infty}(\mathbb{R}^d))$. In this paper, we investigate the regularity of the density of $\nu$. A systematic study of stationary Fokker-Planck equations was carried out in \cite{BKRS15, BKR01, BKR09, BRS12, BS17, B18}. In particular, if the diffusion coefficients satisfy $a_{ij} \in H^{1,p}_{loc}(\mathbb{R}^d)$ for all $1 \leq i,j \leq d$ with $p \in (d, \infty)$, $A = (a_{ij})_{1 \leq i,j \leq d}$ is locally uniformly strictly elliptic and bounded on $\mathbb{R}^d$ (i.e. \eqref{locellst} holds for $U=\mathbb{R}^d$), $\mathbf{G} \in L^p_{loc}(\mathbb{R}^d, \mathbb{R}^d)$ and \eqref{stafokkerp} holds, then according to \cite[Corollary 1.6.9]{BKRS15}, there exists a function $\rho \in H^{1,p}_{loc}(\mathbb{R}^d) \cap C(\mathbb{R}^d)$ such that  $\nu = \rho \,dx$.\\
Stationary Fokker-Planck equations can be viewed as a class of partial differential equations in which the second-order derivatives act on test functions. The regularity theory for their solutions has developed through successive advancements, beginning with the classical Weyl lemma \cite{W40}, and subsequently extended in works such as \cite{S73, F77, BS17}. For further details, we refer the reader to \cite{L25bv} and the references therein. A recent result (\cite[Corollary 1.2]{L25bv}) concerning the regularity of solutions to stationary Fokker-Planck equations generalizes \cite[Corollary 1.6.9]{BKRS15} under the assumption that (cf. Definitions \ref{defndivmat}, \ref{defnvmo})
${\rm div} A \in L^p_{loc}(\mathbb{R}^d, \mathbb{R}^d)$ with $p \in (d, \infty)$, $a_{ij} \in VMO_{loc}(\mathbb{R}^d)$ for all $1 \leq i,j \leq d$, and that the matrix $A = (a_{ij})_{1 \leq i,j \leq d}$ is locally uniformly strictly elliptic and bounded on $\mathbb{R}^d$. If, in addition, $\mathbf{G} \in L^p_{loc}(\mathbb{R}^d, \mathbb{R}^d)$ and \eqref{stafokkerp} holds, then it has been shown in \cite[Corollary 1.2]{L25bv} that there exists a function $\rho \in H^{1,p}_{loc}(\mathbb{R}^d) \cap C(\mathbb{R}^d)$ such that  $\nu = \rho\,dx$.\\
However, the assumption $\mathbf{G} \in L_{loc}^p(\mathbb{R}^d, \mathbb{R}^d)$ may still be regarded as rather restrictive. It is therefore natural to consider relaxing this condition to the case $\mathbf{G} \in L^2_{loc}(\mathbb{R}^d, \mathbb{R}^d)$. In fact, as shown in \cite{F81, St99}, a stochastic process associated with the operator $L$ can be constructed even when the drift is merely locally $L^2$-integrable with respect to a reference measure. This observation motivates the extension of the results in \cite[Corollary 1.6.9]{BKRS15} and \cite[Corollary 1.2]{L25bv} to this more general setting.
In particular, when the diffusion matrix is the identity, the assumption $\mathbf{G} \in L^2_{loc}(\mathbb{R}^d, \mathbb{R}^d)$ has been considered in \cite[Theorem 1(ii)]{BKR97}, where corresponding regularity results for the density of the solution were established. \\
In our setting, however, we require the a priori assumption that the solution $\nu$ admits a density belonging to $L_{loc}^\infty(U)$.
\begin{theo} \label{maintheore}
Let $U \subset \mathbb{R}^d$ be a (possibly unbounded) open set with $d \geq 2$. Assume that $\mathbf{G} \in L^2_{loc}(U, \mathbb{R}^d)$ and $c \in L^1_{loc}(U)$. Let $A = (a_{ij})_{1 \leq i,j \leq d}$ be a symmetric matrix of functions such that $a_{ij} \in VMO_{loc}(U)$ for all $1 \leq i,j \leq d$ (see Definition~\ref{defnvmo}), and suppose that ${\rm div} A \in L^2_{loc}(U, \mathbb{R}^d)$ (see Definition~\ref{defndivmat}). Also assume that $A$ is locally uniformly strictly elliptic and bounded on $U$, i.e. for every open ball $V$ with $\overline{V} \subset U$, there exist constants $\lambda_V, M_V > 0$ such that
\begin{equation} \label{locellst}
\lambda_V \| \xi \|^2 \leq \langle A(x) \xi, \xi \rangle, \quad \;\; \max_{1 \leq i,j \leq d} |a_{ij}(x)| \leq M_V
\quad \text{for a.e. } x \in V \text{ and all } \xi \in \mathbb{R}^d.
\end{equation}
Let $\tilde{\mathbf{F}} \in L^2_{loc}(U, \mathbb{R}^d)$ and $\tilde{f} \in L^1_{loc}(U)$. Suppose that $h \in L^{\infty}_{loc}(U)$ satisfies
\[
\int_U \left({\rm trace}(A \nabla^2 \varphi) + \langle \mathbf{G}, \nabla \varphi \rangle + c \varphi \right) h \, dx 
= \int_U \tilde{f} \varphi \, dx + \int_U \langle \tilde{\mathbf{F}}, \nabla \varphi \rangle \, dx \quad \text{for all } \varphi \in C_0^{\infty}(U).
\]
Then $h \in H^{1,2}_{loc}(U)$.
\end{theo}
The main idea of this work is to construct a constant $\lambda_0 > 0$ and a sub-Markovian resolvent $\left(R_{\alpha}\right)_{\alpha \geq \lambda_0}$ associated with the principal elliptic operator $L^A = {\rm trace}(A \nabla^2)$, and to show that the function $h$ can be represented as the local weak $H^{1,2}$-limit of the family $(\alpha R_{\alpha} h)_{\alpha \geq \lambda_0}$ as $\alpha \rightarrow \infty$. This idea is inspired by \cite[Theorem 2.1]{St99}. In that result, however, the diffusion coefficients $a_{ij}$ are assumed to be locally Hölder continuous for all $1 \leq i,j \leq d$, and the density of the solution measure $\nu$ is assumed to belong to $L^{\infty}(\mathbb{R}^d)$. By contrast, in the present work, we establish resolvent estimates under the weaker assumptions that the coefficients belong to $VMO_{loc}(\mathbb{R}^d)$ and that ${\rm div} A \in L^2_{loc}(\mathbb{R}^d, \mathbb{R}^d)$, which enable us to derive the main result, Theorem \ref{maintheore}.\\
It is worth noting that determining the extent to which the assumptions in our main theorem can be further relaxed remains an open and challenging problem. In particular, it is unclear whether the assumptions that $a_{ij} \in VMO_{loc}(U)$ for all $1 \leq i,j \leq d$ and ${\rm div} A \in L^2_{loc}(U, \mathbb{R}^d)$ can be weakened. Similarly, whether the a priori condition $h \in L^{\infty}_{loc}(U)$ can be replaced by the weaker assumption $h \in L^2_{loc}(U)$ remains open. Furthermore, one may ask whether the result could still hold under the sole assumptions that the solution measure $\nu$ is locally finite and that $Lf \in L^1_{loc}(U, \nu)$ for every $f \in C_0^{\infty}(U)$. Understanding the regularity of solutions to such stationary Fokker-Planck equations is expected to play a crucial role in the study of invariant measures associated with stochastic differential equations having singular coefficients. \\
This paper is organized as follows. Section \ref{sec2} introduces the basic notations and definitions used throughout this paper. Section \ref{sec3} presents the resolvent approaches to elliptic regularity and provides the proof of the main result, Theorem \ref{maintheore}, at the end of the section. Section \ref{sec4} discusses applications to uniqueness results based on the main result and presents some concrete examples.

\section{Notations and definitions} \label{sec2}
For $a,b \in \mathbb{R}$ with $b \leq a$, we write
$a \vee b= b \vee a :=a$,  $a \wedge b = b \wedge a :=b$, $a^+:= a \vee 0$,  $a^-=(-a) \vee 0$.
Let $\mathbb{R}^d$ denote the $d$-dimensional Euclidean space with $d \geq 1$, equipped with the standard inner product $\langle \cdot, \cdot \rangle$ and the corresponding Euclidean norm $\|\cdot\|$. For $A \subset \mathbb{R}^d$, $1_A$ denotes the indicator function of the set $A$, defined by $1_A(x) = 1$ if $x \in A$, and $1_A(x) = 0$ if $x \in \mathbb{R}^d \setminus A$. For any $x_0 \in \mathbb{R}^d$ and $r > 0$, we denote by $B_r(x_0) := \{ x \in \mathbb{R}^d : \|x - x_0\| < r \}$ the open ball centered at $x_0$ with radius $r$. 
The Lebesgue measure on $\mathbb{R}^d$ is denoted by $dx$.  We denote by $\mathcal{B}(\mathbb{R}^d)$ the Borel $\sigma$-algebra on $\mathbb{R}^d$, i.e. the smallest $\sigma$-algebra containing all open subsets of $\mathbb{R}^d$. Let $U$ be an open subset of $\mathbb{R}^d$.
The set of continuous functions on $U$ is denoted by $C(U)$. 
We write $C_0(U)$ to mean the set of continuous functions with compact support in $U$. For $k \in \mathbb{N} \cup \{\infty\}$, the space $C^k(U)$ consists of all functions on $U$ that are $k$-times continuously differentiable. We set $C^k_0(U) := C^k(U) \cap C_0(U)$. 
For $s \in [1,\infty]$, we denote by $L^s(U)$ the space of $dx$-integrable functions on $U$ with exponent $s$, equipped with the standard $L^s$-norm. The space $L^s(U, \mathbb{R}^d)$ refers to the set of vector-valued functions $\mathbf{F}$ with values in $\mathbb{R}^d$ such that $\|\mathbf{F}\| \in L^s(U)$, and we simply write $\|\mathbf{F} \|_{L^s(U)}:=\| \| \mathbf{F} \| \|_{L^s(U)}$. The localized versions, denoted by $L^s_{loc}(U)$ and $L^s_{loc}(U, \mathbb{R}^d)$, consist of functions and vector fields whose restrictions to compact subsets $V$ of $U$ belong to the corresponding $L^s(V)$ and $L^s(V, \mathbb{R}^d)$, respectively.
Given a function $f : U \to \mathbb{R}$, if the weak partial derivative in the $i$-th coordinate direction exists, it is denoted by $\partial_i f$. The Sobolev space $H^{1,s}(U)$ consists of all functions $f \in L^s(U)$ such that each $\partial_i f \in L^s(U)$ for $i = 1, \ldots, d$, equipped with the usual Sobolev norm. The space $H^{1,2}_0(U)$ is defined as the closure of $C_0^\infty(U)$ in $H^{1,2}(U)$. Similarly, $H^{2,s}(U)$ is the Sobolev space of functions $f \in L^s(U)$ such that all first and second-order weak derivatives belong to $L^s(U)$. For such a function $f$, the weak Laplacian is given by $\Delta f := \sum_{i=1}^d \partial_i \partial_i f$, and the weak Hessian matrix is denoted by $\nabla^2 f := (\partial_i \partial_j f)_{1 \leq i,j \leq d}$. 
Finally, for a real $d \times d$ matrix of functions $B = (b_{ij})_{1 \leq i,j \leq d}$, we define its trace as ${\rm trace}(B) := \sum_{i=1}^d b_{ii}$. 
Thus, it holds that
$$
{\rm trace}(B \nabla^2 f) = \sum_{i,j=1}^d b_{ij} \partial_i \partial_j f =\sum_{i,j=1}^d \left(\frac{b_{ij}+b_{ji}}{2} \right) \partial_i \partial_j f.
$$
\begin{defn} \label{defndivmat}
Let $U$ be an open subset of $\mathbb{R}^d$, $\mathbf{E} = (e_1, \dots, e_d) \in L^1_{loc}(U, \mathbb{R}^d)$ and $B = (b_{ij})_{1 \leq i,j \leq d}$ be a (possibly non-symmetric) matrix of functions with $b_{ij} \in L^1_{loc}(U)$ for all $1 \leq i,j \leq d$. Let $\mathbf{F} \in L^1_{loc}(U, \mathbb{R}^d)$ and $f \in L^1_{loc}(U)$.
\begin{itemize}
\item[(i)] We say that ${\rm div} B = \mathbf{E}$ if
\[
\int_U \sum_{i,j=1}^d b_{ij} \, \partial_i \phi_j \, dx = - \int_U \sum_{j=1}^d e_j \, \phi_j \, dx \quad \text{for all } (\phi_1, \dots, \phi_d) \in C_0^\infty(U)^d.
\]
\item[(ii)] We say that ${\rm div} \mathbf{F} = f$ if
\[
\int_U \langle \mathbf{F}, \nabla \varphi \rangle \, dx = - \int_U f \varphi \, dx \quad \text{for all } \varphi \in C_0^\infty(U).
\]
\end{itemize}
\end{defn}
\begin{defn} \label{defnvmo}
Let $\omega$ be a positive continuous function on $[0,\infty)$ with $\omega(0) = 0$. The space $VMO_\omega$ consists of all functions $g \in L^1_{loc}(\mathbb{R}^d)$ such that
\[
\sup_{z \in \mathbb{R}^d,\, r < R} \, r^{-2d} \int_{B_r(z)} \int_{B_r(z)} |g(x) - g(y)| \, dx \, dy \leq \omega(R) \quad \text{for all } R \in (0,\infty).
\]
Let $B$ be an open ball in $\mathbb{R}^d$. We define $VMO(B)$ as the set of all functions $f \in L^1(B)$ for which there exists an extension $\tilde{f} \in VMO_{\omega}$  (for some modulus $\omega$) such that $\tilde{f}|_B = f$. For an open set $U \subset \mathbb{R}^d$, the space $VMO_{loc}(U)$ is defined as the collection of all functions $f \in L^1_{loc}(U)$ such that $f \in VMO(B)$ for every open ball $B \subset \mathbb{R}^d$ with $\overline{B} \subset U$.
\end{defn}
It is well known that
\[
C_0(\mathbb{R}^d) \cup H^{1,d}(\mathbb{R}^d) \subset VMO_{\omega}
\quad \text{(see \cite[Proposition 2.1]{L23ID})}.
\]
Moreover, using standard extension arguments, one obtains
\[
C(U) \cup H^{1,d}_{loc}(U) \subset VMO_{loc}(U)
\quad \text{for any open set } U \subset \mathbb{R}^d.
\]
By the definition of $VMO_{loc}(U)$, we have that if $f, g \in VMO_{loc}(U)$, then $f + g \in VMO_{loc}(U)$. Moreover, \cite[Proposition 5.3]{L25bv} shows that $VMO_{loc}(U) \cap L^{\infty}_{loc}(U)$ is closed under pointwise multiplication.

\section{Resolvent approaches to elliptic regularity} \label{sec3} 
In this section, we mainly consider the following condition {\bf (Hypo)} in Theorems \ref{maintheolem} and \ref{submaintheor}. \\ \\
{\bf (Hypo)}: $d \geq 2$, and $A=(a_{ij})_{1 \leq i,j \leq d}$ is a symmetric matrix of functions on $\mathbb{R}^d$, satisfying the following condition for some constants $\lambda, M > 0$:
\begin{equation} \label{ellipticit}
\lambda \| \xi \|^2 \leq \langle A(x) \xi, \xi \rangle, \quad \;\; \max_{1 \leq i,j \leq d} |a_{ij}(x)| \leq M
\quad \text{ for a.e. } x \in \mathbb{R}^d \text{ and all } \xi \in \mathbb{R}^d.
\end{equation}
Moreover, for all $1 \leq i,j \leq d$, $a_{ij} \in VMO_{\omega}$, where $\omega$ is a positive continuous function on $[0, \infty)$ with $\omega(0) = 0$.

\begin{lem} \label{auxilemplus}
Let $v \in H^{2,1}(U) \cap H^{1,2}(U) \cap L^{\infty}(U)$, $c \in L^1(U)$  and $A=(a_{ij})_{1 \leq i,j \leq d}$ be a (possibly non-symmetric) matrix of smooth functions such that
\begin{equation} \label{varidenegat}
\int_{U} \langle A \nabla v, \nabla \varphi \rangle +\langle  {\rm div}A, \nabla v \rangle \varphi + cv \varphi \,dx  \leq 0 \quad \text{ for all $\varphi \in C_0^{\infty}(U)$ with $\varphi \geq 0$}.
\end{equation}
Then, 
$$
\int_{U} \langle A \nabla v^+, \nabla \varphi \rangle +\langle  {\rm div}A, \nabla v^+ \rangle \varphi + cv^+ \varphi \,dx  \leq 0 \quad \text{ for all $\varphi \in C_0^{\infty}(U)$ with $\varphi \geq 0$}.
$$
\end{lem}
\begin{proof}
Define the partial differential operator $\mathcal{L}$ by
$$
\mathcal{L}f := -{\rm trace}(A \nabla^2 f) + cf, \quad f \in H^{2,1}(U) \text{ with } \text{supp}(f) \subset U.
$$
Applying integration by parts to \eqref{varidenegat}, we obtain
$$
\int_{U}\mathcal{L}v \cdot \varphi\,dx \leq 0\quad \text{for all } \varphi \in C_0^{\infty}(U) \text{ with } \varphi \geq 0,
$$
and hence
\begin{equation} \label{operatorlnega}
\mathcal{L}v \leq 0 \quad \text{in } U.
\end{equation}
Define
\[
f_\varepsilon(t) := 
\begin{cases}
\sqrt{t^2 + \varepsilon^2} - \varepsilon & \text{if } t \geq 0, \\
0 & \text{if } t < 0.
\end{cases}
\]
Then $f_\varepsilon \in C^1(\mathbb{R})$, and $f'_\varepsilon \in H^{1,\infty}(\mathbb{R}) \cap C(\mathbb{R})$. Moreover, we have
\[
f'_\varepsilon(t) = 
\begin{cases}
\dfrac{t}{\sqrt{t^2 + \varepsilon^2}} & \text{if } t \geq 0, \\
0 & \text{if } t < 0,
\end{cases}
\quad \text{and} \quad
f''_\varepsilon(t) = 
\begin{cases}
\dfrac{\varepsilon^2}{(t^2 + \varepsilon^2)^{3/2}} & \text{if } t > 0, \\
0 & \text{if } t < 0.
\end{cases}
\]
It then follows that
\begin{equation} \label{fepsilonletze}
\lim_{\varepsilon \to 0^+} f_\varepsilon(t) = t^+, \quad \text{and} \quad \lim_{\varepsilon \to 0^+} f'_\varepsilon(t) =1_{(0,\infty)}(t) \quad \text{for all } t \in \mathbb{R}.
\end{equation}
Meanwhile, by \cite[Theorem 4.4]{EG15}, we have $v^+ \in H^{1,2}(U)$ and $\nabla v^+ = 1_{(0, \infty)} (v) \nabla v$ in $U$. Therefore, by the Lebesgue dominated convergence theorem and the chain rule,
\begin{equation} \label{approxih12}
\lim_{\varepsilon \rightarrow 0+} f_{\varepsilon}(v) = v^+ \quad \text{in } H^{1,2}(U).
\end{equation}
Let $\phi$ be a standard mollifier on $\mathbb{R}$, and for each $k \in \mathbb{N}$, define $\phi_k(t) = k\phi(kt)$ for $t \in \mathbb{R}$.
Let \( f_{\varepsilon,k} := f_\varepsilon * \phi_k \in C^{\infty}(\mathbb{R}) \). Then,
\begin{equation} \label{fvarepksign}
 f_{\varepsilon,k}'(t) \geq 0 \quad \text{and} \quad f_{\varepsilon,k}''(t) \geq 0 \quad \text{for all } t \in \mathbb{R}.
\end{equation}
Moreover, since $f_{\varepsilon}$ and $f'_{\varepsilon}$ are continuous on $\mathbb{R}$ and $v \in L^{\infty}(U)$, we obtain
\begin{equation} \label{letingkfkep}
\lim_{k \to \infty} f_{\varepsilon,k}(v) = f_\varepsilon(v) \quad \text{and} \quad 
\lim_{k \to \infty} f_{\varepsilon,k}'(v) = f_\varepsilon'(v) \quad \text{uniformly on } U.
\end{equation}
Therefore, by the Lebesgue dominated convergence theorem and the chain rule,
\begin{equation} \label{letingkfkepnex}
\lim_{k \to \infty} f_{\varepsilon, k}(v) = f_{\varepsilon}(v) \quad \text{in } H^{1,2}(U).
\end{equation}
Applying integration by parts and using \eqref{operatorlnega} and \eqref{fvarepksign}, we find
\begin{align*}
&\int_{U} \langle A \nabla f_{\varepsilon,k}(v), \nabla \varphi \rangle + \langle {\rm div}A, \nabla f_{\varepsilon,k}(v) \rangle \varphi + cf_{\varepsilon,k}(v) \varphi \,dx  \\
&= \int_{U} \Big( -f'_{\varepsilon, k}(v){\rm trace}(A \nabla^2 v) - f_{\varepsilon,k}''(v)\langle A \nabla v, \nabla v \rangle + cf_{\varepsilon,k}(v) \Big) \varphi \,dx \\
&= \int_{U} f'_{\varepsilon, k}(v) \mathcal{L}v \cdot \varphi - f_{\varepsilon,k}''(v)\langle A \nabla v, \nabla v \rangle \varphi + \big( -c f_{\varepsilon, k}'(v) v + c f_{\varepsilon, k}(v) \big) \varphi\,dx \\
&\leq \int_{U} \big( -c f_{\varepsilon, k}'(v) v + c f_{\varepsilon, k}(v) \big) \varphi \,dx \quad \text{for all } \varphi \in C_0^{\infty}(U) \text{ with } \varphi \geq 0.
\end{align*}
Letting $k \rightarrow \infty$, it follows from \eqref{letingkfkepnex} and \eqref{letingkfkep} that
\begin{align*}
&\int_{U} \langle A \nabla f_{\varepsilon}(v), \nabla \varphi \rangle + \langle {\rm div}A, \nabla f_{\varepsilon}(v) \rangle \varphi + cf_{\varepsilon}(v) \varphi \,dx \\
&\leq \int_{U} \big( -c f_{\varepsilon}'(v) v + c f_{\varepsilon}(v) \big) \varphi\,dx \quad \text{for all } \varphi \in C_0^{\infty}(U) \text{ with } \varphi \geq 0.
\end{align*}
Finally, letting $\varepsilon \rightarrow 0^+$ and applying \eqref{approxih12}, \eqref{fepsilonletze}, and the Lebesgue dominated convergence theorem, we conclude
\begin{align*}
&\int_{U} \langle A \nabla v^+, \nabla \varphi \rangle + \langle {\rm div}A, \nabla v^+ \rangle \varphi + cv^+ \varphi \,dx \\
&\qquad \leq \int_{U} \big( -c 1_{(0, \infty)}(v) v + cv^+ \big) \varphi\,dx = 0 \quad \text{for all } \varphi \in C_0^{\infty}(U) \text{ with } \varphi \geq 0,
\end{align*}
as desired.
\end{proof}

\begin{theo} \label{maintheolem}
Assume {\bf (Hypo)}. Let $B$ be an open ball in $\mathbb{R}^d$ and let $f \in L^s(B)$ with $s \in (1, \infty)$. Define the second-order partial differential operator $L^A$ by
\begin{equation} \label{nondivpardif}
L^A u := {\rm trace} (A \nabla^2 u), \quad u \in H^{2,1}(B).
\end{equation}
Then the following assertions hold:
\begin{itemize}
\item[(i)]
There exists a constant $\lambda_0 \in (0, \infty)$ depending only on 
$d, \lambda, M, B, s$ and $\omega$ such that  for every $\alpha \geq \lambda_0$, there exists a unique function $R_{\alpha} f \in H^{2,s}(B) \cap H^{1,s}_0(B)$ satisfying
\begin{equation} \label{operatorigor}
(\alpha - L^A) R_{\alpha} f = f \quad \text{in } B.
\end{equation}
Moreover, for such $\alpha \geq \lambda_0$, there exists a constant $C > 0$ depending only on $d, \lambda, M, B, s, \omega$ and $\alpha$ such that
$$
\|R_{\alpha} f\|_{H^{2,s}(B)} \leq C\|f\|_{L^s(B)}.
$$

\item[(ii)]
Assume $f \in L^{\infty}(B)$ and $s \in (d, \infty)$. Let $\lambda_0 > 0$ be as in (i), and let $\alpha > \lambda_0$. Then the solution $R_{\alpha} f \in H^{2,s}(B) \cap C(\overline{B})$ obtained in (i) satisfies
$$
\|\alpha R_{\alpha} f\|_{L^{\infty}(B)} \leq \|f\|_{L^{\infty}(B)}.
$$
\end{itemize}
\end{theo}
\begin{proof}
(i) The assertion follows from \cite[Theorem 8]{DK11} (cf. \cite{V92}). \\ \\
(ii) Let $\eta$ be a standard mollifier on $\mathbb{R}^d$, and for each $n \geq 1$, define $\eta_n \in C_0^{\infty}(B_{1/n})$ by $\eta_n(x) := n^d \eta(nx)$ for $x \in \mathbb{R}^d$. \\ \\
{\sf \underline{Claim}:} If $g \in VMO_{\omega}$, then $g * \eta_n \in VMO_{\omega}$ for all $n \geq 1$. \\ \\
To prove the claim, let $n \geq 1, z \in \mathbb{R}^d$, $R > 0$ and $r \in (0, R)$. % and $x, y \in B_r(z)$. 
Then,
\begin{align*}
&r^{-2d}\int_{B_r(z)}\int_{B_r(z)}|g* \eta_n(x)-g*\eta_n(y)| \,dxdy \\
&= r^{-2d}\int_{B_r(z)}\int_{B_r(z)} \left| \int_{\mathbb{R}^d} \left(g(x-\xi)- g(y-\xi) \right) \eta_n(\xi)	d\xi \right| \,dxdy \\
& \leq \int_{\mathbb{R}^d}\left( r^{-2d}\int_{B_r(z)} \int_{B_r(z)}\left| g(x-\xi)- g(y-\xi) \right| \,dx dy \right) \eta_n(\xi)	d\xi \\
& = \int_{\mathbb{R}^d}\left( r^{-2d}\int_{B_r(z-\xi)} \int_{B_r(z-\xi)}\left| g(x')- g(y') \right| \,dx' dy' \right) \eta_n(\xi)	d\xi \\
&\leq \int_{\mathbb{R}^d} \omega(R) \eta_n(\xi) d\xi = \omega(R).
\end{align*}
Therefore,
\[
\sup_{z \in \mathbb{R}^d,\, r < R} \, r^{-2d} \int_{B_r(z)} \int_{B_r(z)} |g*\eta_n(x) - g*\eta_n(y)| \, dx \, dy \leq \omega(R) \quad \text{for all } R \in (0,\infty),
\]
and thus the claim is proved. \\ \\
Now, for each $n \geq 1$ and $1 \leq i,j \leq d$, define $a_{n, ij} := a_{ij} * \eta_n \in C^{\infty}(\mathbb{R}^d)$. Then, by the claim, $a_{n, ij} \in VMO_{\omega}$ for all $n \geq 1$ and $1 \leq i,j \leq d$. Let $A_n := (a_{n, ij})_{1 \leq i,j \leq d}$, $n \geq 1$. Then, it holds that
\[
\lambda \| \xi \|^2 \leq \langle A_n(x) \xi, \xi \rangle, \quad \;\; \max_{1 \leq i,j \leq d} |a_{n, ij}(x)| \leq M
\quad \text{ for a.e. } x \in \mathbb{R}^d \text{ and all } \xi \in \mathbb{R}^d
\]
and
\begin{equation} \label{almoconver}
\lim_{n \rightarrow \infty} a_{n, ij} = a_{ij} \quad \text{a.e. in } \mathbb{R}^d \quad \text{for all } 1 \leq i,j \leq d.
\end{equation}
Now, let $s \in (d, \infty)$, $f \in L^{\infty}(B)$, and let $\lambda_0 > 0$ be the constant from part (i), with $\alpha > \lambda_0$. Then, by (i), there exist unique functions $R_{n, \alpha} f, R_{\alpha} f \in H^{2,s}(B) \cap H^{1,2}_0(B) \cap C(\overline{B})$ satisfying
\begin{equation} \label{resolequn}
(\alpha - L^{A_n}) R_{n, \alpha} f = f \quad \text{in } B,
\end{equation}
and also \eqref{operatorigor} holds. Moreover, by (i), there exists a constant $\tilde{C} > 0$ depending only on $d, \lambda, M, B, s, \omega, \alpha$ (independent of $n$ and $f$) such that
\begin{equation*} %\label{basich2estimate}
\|R_{n, \alpha} f\|_{H^{2,s}(B)} \leq \tilde{C}\|f\|_{L^s(B)},
\end{equation*}
\begin{equation} \label{basich2estorigi}
\|R_{\alpha} f\|_{H^{2,s}(B)} \leq \tilde{C}\|f\|_{L^s(B)}.
\end{equation}
Applying integration by parts to \eqref{resolequn}, we get
\begin{equation}\label{rnalphaequa}
\int_{B} \langle A_n \nabla R_{n, \alpha} f, \nabla \varphi \rangle + \langle {\rm div}A_n, \nabla R_{n, \alpha}f \rangle \varphi + (\alpha R_{n, \alpha} f) \varphi \,dx = \int_{B} f \varphi \,dx \quad \text{ for all $\varphi \in C_0^{\infty}(B)$.}
\end{equation}
Define $v := R_{n, \alpha} f - \frac{1}{\alpha} \|f\|_{L^{\infty}(B)} \in H^{2,s}(B) \cap H^{1,2}(B) \cap C(\overline{B})$. Then $v = -\frac{1}{\alpha} \|f\|_{L^{\infty}(B)}$ on $\partial B$. Substituting $R_{n, \alpha} f = v + \frac{1}{\alpha} \|f\|_{L^{\infty}(B)}$ into \eqref{rnalphaequa} gives
\begin{align*}
&\int_{B} \langle A_n \nabla v, \nabla \varphi \rangle + \langle {\rm div}A_n, \nabla v \rangle \varphi + \alpha v \varphi \,dx \\
& \qquad = \int_{B} (f - \|f\|_{L^{\infty}(B)}) \varphi \,dx \leq 0 \quad \text{ for all $\varphi \in C_0^{\infty}(B)$ with $\varphi \geq 0$.}
\end{align*}
By \cite[Theorem 4.4(iii)]{EG15}, $v^+ \in H^{1,2}(B) \cap C(\overline{B})$ with $v^+ = 0$ on $\partial B$, hence $v^+ \in H^{1,2}_0(B)$. Then by Lemma \ref{auxilemplus},
\begin{align*}
\int_{B} \langle A_n \nabla v^+, \nabla \varphi \rangle + \langle {\rm div}A_n, \nabla v^+ \rangle \varphi + \alpha v^+ \varphi \,dx \leq 0 \quad \text{ for all $\varphi \in C_0^{\infty}(B)$ with $\varphi \geq 0$. }
\end{align*}
By the weak maximum principle \cite[Theorem 2.1.8]{BKRS15} (cf. \cite{Tma}), it follows that $v^+ \leq 0$ in $B$, and hence $v \leq 0$ in $B$. Thus,
\begin{equation} \label{submarkresolv}
R_{n, \alpha} f \leq \frac{1}{\alpha} \|f\|_{L^{\infty}(B)} \quad \text{in } B.
\end{equation}
Similarly, define $w := -R_{n, \alpha} f - \frac{1}{\alpha} \|f\|_{L^{\infty}(B)} \in H^{2,2}(B) \cap H^{1,2}(B) \cap C(\overline{B})$. Then $w = -\frac{1}{\alpha} \|f\|_{L^{\infty}(B)}$ on $\partial B$. By \cite[Theorem 4.4(iii)]{EG15}, $w^+ \in H^{1,2}(B) \cap C(\overline{B})$ with $w^+ = 0$ on $\partial B$, hence $w^+ \in H^{1,2}_0(B)$. 
Substituting $R_{n, \alpha} f = -w - \frac{1}{\alpha} \|f\|_{L^{\infty}(B)}$ into \eqref{rnalphaequa} yields
\begin{align*}
&\int_{B} \langle A_n \nabla w, \nabla \varphi \rangle + \langle {\rm div}A_n, \nabla w \rangle \varphi + \alpha w \varphi \,dx \\
& \qquad = \int_{B} (-f - \|f\|_{L^{\infty}(B)}) \varphi \,dx \leq 0 \quad \text{ for all $\varphi \in C_0^{\infty}(B)$ with $\varphi \geq 0$.}
\end{align*}
Similarly to the above, applying Lemma \ref{auxilemplus} and the weak maximum principle \cite[Theorem 2.1.8]{BKRS15} gives $w \leq 0$ in $B$, hence
\begin{equation} \label{submarkresolw}
-\frac{1}{\alpha} \|f\|_{L^{\infty}(B)} \leq R_{n, \alpha} f \quad \text{in } B.
\end{equation}
Combining \eqref{submarkresolv} and \eqref{submarkresolw}, we obtain
\begin{equation} \label{submarkestimresol}
\|R_{n, \alpha} f \|_{L^{\infty}(B)} \leq \frac{1}{\alpha} \|f\|_{L^{\infty}(B)}.
\end{equation}
Meanwhile, by \eqref{operatorigor} and \eqref{resolequn},
\begin{equation} \label{resunsubtrac}
(\alpha - L^{A_n}) (R_{\alpha} f - R_{n, \alpha} f) = \sum_{i,j=1}^d (a_{ij} - a_{n, ij}) \partial_i \partial_j R_{\alpha} f \quad \text{in } B.
\end{equation}
Applying (i) to \eqref{resunsubtrac}, we get
\begin{equation} \label{estimapple}
\| R_{\alpha} f - R_{n, \alpha} f \|_{H^{2,s}(B)} \leq \tilde{C} \sum_{i,j=1}^d \| (a_{ij} - a_{n, ij}) \partial_i \partial_j R_{\alpha} f \|_{L^s(B)}.
\end{equation}
Using \eqref{basich2estorigi},
\[
\| \partial_i \partial_j R_{\alpha} f \|_{L^s(B)} \leq \tilde{C}\|f\|_{L^s(B)}.
\]
Hence, by \eqref{almoconver} and the Lebesgue dominated convergence theorem applied to \eqref{estimapple},
\begin{equation} \label{h2converge}
\lim_{n \rightarrow \infty} R_{n, \alpha} f = R_{\alpha} f \quad \text{in } H^{2,s}(B).
\end{equation}
Finally, by \eqref{h2converge} and the Banach–Alaoglu theorem (cf. \cite[Theorem 3.16]{Br11}) applied to \eqref{submarkestimresol}, there exists a subsequence of $(R_{n, \alpha} f)_{n \geq 1}$ (still denoted by the same) such that $R_{n, \alpha} f \to R_{\alpha} f$ weakly-* in $L^{\infty}(B)$, and hence by \cite[Proposition 3.13(iii)]{Br11},
\[
\| R_{\alpha} f \|_{L^{\infty}(B)} \leq \liminf_{n \to \infty} \| R_{n, \alpha} f \|_{L^{\infty}(B)} \leq \frac{1}{\alpha} \|f\|_{L^{\infty}(B)}.
\]
This completes the proof.
\end{proof}

\begin{lem} \label{coercievforml}
Let $A = (a_{ij})_{1 \leq i,j \leq d}$ be a (possibly non-symmetric) matrix of functions on $\mathbb{R}^d$ satisfying \eqref{ellipticit}. Let $B$ be an open ball, and define 
$(\mathcal{E}^0, D(\mathcal{E}^0))$ as the coercive form obtained as the closure of
\begin{equation} \label{closablema}
\mathcal{E}^0(f, g) = \int_{B} \langle A \nabla f, \nabla g \rangle \,dx, \quad f, g \in C_0^{\infty}(B).
\end{equation}
Let $(L^0, D(L^0))$ denote the generator associated with $(\mathcal{E}^0, D(\mathcal{E}^0))$. Then, the following hold:
\begin{itemize}
\item[(i)]
$D(\mathcal{E}^0) = H^{1,2}_0(B)$.
\item[(ii)]
Let $p \in (d, \infty)$ and assume ${\rm div} A \in L^2(B, \mathbb{R}^d)$. If $f \in H^{2,p}(B) \cap H^{1,2}_0(B)$, then $f \in D(L^0)$ and
\begin{equation} \label{claimediden}
L^0 f = {\rm trace}(A \nabla^2 f) + \langle {\rm div}A, \nabla f \rangle.
\end{equation}
\end{itemize}
\end{lem}
\begin{proof}
(i) By \cite[Theorem 3.2]{L25bv}, the bilinear form $(\mathcal{E}^0, C_0^{\infty}(B))$ is closable, and its closure $(\mathcal{E}^0, D(\mathcal{E}^0))$ is a Dirichlet form. Moreover, by \cite[Theorem 3.2(ii)]{L25bv}, $D(\mathcal{E}^0) = H^{1,2}_0(B)$ follows. \\ \\
(ii) Let $f \in H^{2,p}(B) \cap H^{1,2}_0(B)$ for some $p \in (d, \infty)$. Then, by the Sobolev embedding theorem, $\nabla f \in L^{\infty}(B)$. By \cite[Proposition 5.6]{L25bv}, we have
$$
{\rm div}(A \nabla f) = {\rm trace}(A \nabla^2 f) + \langle {\rm div}A, \nabla f \rangle.
$$
Let $\varphi \in C_0^{\infty}(B)$. Then,
\begin{align*}
\mathcal{E}^0(f, \varphi) &= \int_{B} \langle A \nabla f, \nabla \varphi \rangle\,dx 
= -\int_{B} {\rm div}(A \nabla f) \cdot \varphi \,dx \\
&= -\int_{B} \Big( {\rm trace}(A \nabla^2 f) + \langle {\rm div}A, \nabla f \rangle \Big) \varphi\,dx.
\end{align*}
Since ${\rm trace}(A \nabla^2 f) + \langle {\rm div}A, \nabla f \rangle \in L^2(B)$, the function $f$ belongs to $D(L^0)$ and satisfies \eqref{claimediden}. This completes the proof.
\end{proof}

\begin{theo} \label{submaintheor}
Assume {\bf (Hypo)}. Let $x_0 \in \mathbb{R}^d$ and $R \in (0, \infty)$, and suppose that ${\rm div}A \in L^2_{loc}(B_R(x_0), \mathbb{R}^d)$.
Let $\mathbf{G} \in L^2_{loc}(B_R(x_0), \mathbb{R}^d)$, $c \in L^1_{loc}(B_R(x_0))$, $h \in L^{\infty}_{loc}(B_R(x_0))$, $\tilde{f} \in L^1_{loc}(B_R(x_0))$ and $\tilde{\mathbf{F}} \in L^2_{loc}(B_R(x_0), \mathbb{R}^d)$. Assume that
\begin{align} 
&\int_{B_R(x_0)} \left({\rm trace} (A \nabla^2 u) + \langle \mathbf{G}, \nabla u \rangle + cu \right) h\,dx \nonumber\\
& \qquad = \int_{B_R(x_0)} \tilde{f} u \,dx + \int_{B_R(x_0)} \langle \tilde{\mathbf{F}}, \nabla u \rangle\,dx \quad \text{ for all } u \in C_0^{\infty}(B_R(x_0)). \label{doublediverf}
\end{align}
Then $h \in H^{1,2}_{loc}(B_R(x_0))$.
\end{theo}
\begin{proof}
Let $s \in (0, R)$ and choose $r \in (s, R)$. Set $B := B_r(x_0)$.
Take $\chi \in C_0^{\infty}(B)$ such that $\chi \equiv 1$ on $B_s(x_0)$ and $0 \leq \chi \leq 1$ on $\mathbb{R}^d$.
To prove the theorem, it suffices to show that $\chi h \in H^{1,2}_0(B)$.
Define the partial differential operator $L^A$ as in \eqref{nondivpardif}.
Let $p \in (d, \infty)$ and $f \in L^{\infty}(B)$. Then, by Theorem \ref{maintheolem}, there exists a constant $\lambda_0 \in (0, \infty)$ depending only on 
$d, \lambda, M, B, p$ and $\omega$ such that if $\alpha \geq  \lambda_0$, then there exists a unique function $R_{\alpha} f \in H^{2,p}(B) \cap H^{1,2}_0(B) \cap C(\overline{B})$ satisfying
\begin{equation} \label{operatorequ}
(\alpha - L^A) R_{\alpha} f = f \quad \text{in } B,
\end{equation}
\begin{equation} \label{submarkoestim}
\| \alpha R_{\alpha} f \|_{L^{\infty}(B)} \leq \| f \|_{L^{\infty}(B)},
\end{equation}
and
\begin{equation} \label{h2pestimate}
\| R_{\alpha} f \|_{H^{2,p}(B)} \leq C \| f \|_{L^p(B)},
\end{equation}
where $C > 0$ depends only on $d, \lambda, M, B, p, \omega$ and $\alpha$.
Using mollification, let us approximate the function $h \in L^{\infty}(B)$ by a sequence $(h_n)_{n \geq 1} \subset C_0^{\infty}(B_R(x_0))$ such that
\[
\lim_{n \to \infty} h_n = h \quad \text{a.e. in } B,
\]
and
\begin{equation} \label{fnestimbyh}
\| h_n \|_{L^\infty(B_R(x_0))} \leq \| h \|_{L^\infty(B)} \quad \text{for all } n \geq 1.
\end{equation}
Now fix $\alpha > \lambda_0$.
By applying \eqref{operatorequ} and \eqref{h2pestimate}, along with the Sobolev embedding and the Lebesgue dominated convergence theorem, we deduce that
\begin{equation} \label{ralphfnconv}
R_\alpha h_n \in H^{2,p}(B) \cap H^{1,2}_0(B) \cap C(\overline{B}) \quad \text{for all } n \geq 1, \quad \lim_{n \to \infty} R_\alpha h_n = R_\alpha h \quad \text{ in } H^{2,p}(B)  \text{ and } C(\overline{B}).
\end{equation}
Moreover, from \eqref{submarkoestim} and \eqref{fnestimbyh}, we obtain
\begin{equation} \label{submarestifn}
\| \alpha R_{\alpha} h_n \|_{L^{\infty}(B)} \leq  \|h_n \|_{L^{\infty}(B)} \leq \| h \|_{L^{\infty}(B)}.
\end{equation}
By the product rule and using \eqref{operatorequ}, we have
\begin{align}
L^A(\chi R_\alpha h_n) 
&= R_{\alpha} h_n \cdot L^A \chi + \chi \cdot L^A R_\alpha h_n + \langle 2A \nabla \chi, \nabla R_\alpha h_n \rangle \nonumber \\
&= R_{\alpha} h_n \cdot L^A \chi + \chi (\alpha R_\alpha h_n - h_n) + \langle 2A \nabla \chi, \nabla R_\alpha h_n \rangle. \label{producrule}
\end{align}
Now define the coercive form $(\mathcal{E}^0, D(\mathcal{E}^0))$ as the closure of \eqref{closablema}.
By Lemma \ref{coercievforml}(i) and the ellipticity of $A$, the norm $\| \cdot \|_{D(\mathcal{E}^0)}$ is equivalent to $\| \cdot \|_{H^{1,2}_0(B)}$, and we have $D(\mathcal{E}^0) = H^{1,2}_0(B)$.
Let $(L^0, D(L^0))$ be the generator associated with $(\mathcal{E}^0, D(\mathcal{E}^0))$. Then, by Lemma \ref{coercievforml}(ii), we have
\[
H^{2,p}(B) \cap H^{1,2}_0(B) \subset D(L^0),
\]
and 
\begin{equation} \label{realizatio}
L^0 f = L^A f + \langle {\rm div}A, \nabla f \rangle \quad \text{ for all $f \in H^{2,p}(B) \cap H^{1,2}_0(B)$. }
\end{equation}
Since $\chi R_{\alpha} h_n \in H^{2,p}(B) \cap H^{1,2}_0(B)$, it follows that $\chi R_\alpha h_n \in D(L^0) \subset D(\mathcal{E}^0)$.
Hence, by integration by parts, together with \eqref{realizatio}, \eqref{producrule}, \eqref{fnestimbyh}, and \eqref{submarestifn}, we compute
\begin{align}
&\mathcal{E}^0(\chi \alpha R_\alpha h_n, \chi \alpha R_\alpha h_n)
= -\int_{B} L^0(\chi \alpha R_\alpha h_n) \cdot (\chi \alpha R_\alpha h_n)\, dx \nonumber \\
&= -\int_{B} L^{A}(\chi \alpha R_\alpha h_n) \cdot (\chi \alpha R_\alpha h_n)\, dx
    - \int_{B} \langle {\rm div}A, \nabla(\chi \alpha R_\alpha h_n) \rangle (\chi \alpha R_\alpha h_n)\, dx \nonumber \\
&= -\int_{B} \chi (L^A \chi) \cdot (\alpha R_\alpha h_n)^2\, dx 
    - \alpha \int_{B} (\alpha R_\alpha h_n - h_n) \chi^2 \alpha R_\alpha h_n\, dx \nonumber \\
&\quad \quad - \int_{B} \langle 2A \nabla \chi, \nabla(\chi \alpha R_\alpha h_n) \rangle \alpha R_\alpha h_n\, dx 
    + \int_{B} \langle 2A \nabla \chi, \nabla \chi \rangle (\alpha R_\alpha h_n)^2\, dx  \nonumber \\
&\quad \quad - \int_{B} \langle {\rm div}A, \nabla(\chi \alpha R_\alpha h_n) \rangle (\chi \alpha R_\alpha h_n)\, dx \label{energybound} \\
&\leq \|\chi\|_{L^1(B)} \|L^A \chi\|_{L^\infty(B)} \|h\|_{L^\infty(B)}^2 
    + 2\alpha \|h\|_{L^\infty(B)}^2 \|\chi^2\|_{L^1(B)} \nonumber \\
&\quad \quad + 2\|h\|_{L^\infty(B)} \cdot \mathcal{E}^0(\chi, \chi)^{1/2} \mathcal{E}^0(\chi \alpha R_\alpha h_n, \chi \alpha R_\alpha h_n)^{1/2} 
    + 2\|h\|_{L^\infty(B)}^2 \mathcal{E}^0(\chi, \chi) \nonumber \\
&\quad \quad + \|{\rm div}A\|_{L^2(B)} \|h\|_{L^\infty(B)} \lambda^{-1/2} \mathcal{E}^0(\chi \alpha R_\alpha h_n, \chi \alpha R_\alpha h_n)^{1/2} \nonumber \\
&\leq C_1 + C_2 \cdot \mathcal{E}^0(\chi \alpha R_\alpha h_n, \chi \alpha R_\alpha h_n)^{1/2}, \nonumber
\end{align}
where
\begin{align*}
C_1 &= \|\chi\|_{L^1(B)} \|L^A \chi\|_{L^\infty(B)} \|h\|_{L^\infty(B)}^2 
+ 2\alpha \|h\|_{L^\infty(B)}^2 \|\chi^2\|_{L^1(B)} 
+ 2\|h\|_{L^\infty(B)}^2 \mathcal{E}^0(\chi, \chi), \\
C_2 &= 2\|h\|_{L^\infty(B)} \cdot \mathcal{E}^0(\chi, \chi)^{1/2} 
+ \|{\rm div}A\|_{L^2(B)} \|h\|_{L^\infty(B)} \cdot \lambda^{-1/2}.
\end{align*}
Thus, by Young's inequality, we obtain
\[
\mathcal{E}^0 (\chi \alpha R_{\alpha}h_n, \chi \alpha R_{\alpha}h_n) \leq 2C_1 + C_2^2.
\]
Therefore, by the Banach-Alaoglu theorem and \eqref{ralphfnconv}, we conclude that $\chi \alpha R_{\alpha} h \in D(\mathcal{E}^0)$ and
\begin{equation} \label{weaklyconh}
\lim_{n \rightarrow \infty} \chi \alpha R_{\alpha} h_n = \chi \alpha R_{\alpha} h, \quad \text{ weakly in } D(\mathcal{E}^0).
\end{equation}
Meanwhile, let $g \in H^{2,p}(B)$ with ${\rm supp}(g) \subset B$. Then, by mollification, there exists a sequence $g_n \in C_0^{\infty}(B)$ such that $g_n \to g$ in $H^{2,p}(B)$. Hence, \eqref{doublediverf} implies that
\begin{equation} \label{doublediwstr}
\int_{B} \left( L^A g + \langle \mathbf{G}, \nabla g \rangle + c g \right) h\, dx 
= \int_{B} \tilde{f} g\, dx + \int_{B} \langle \tilde{\mathbf{F}}, \nabla g \rangle\, dx
\end{equation}
for all $g \in H^{2,p}(B)$ with ${\rm supp}(g) \subset B$. We begin by noting the trivial inequality
\[
-\alpha \int_{B} (\alpha R_{\alpha} h - h)^2 \chi^2 dx \leq 0,
\]
and hence
\begin{align} \small
& - \alpha \int_{B} (\alpha R_{\alpha} h - h) \cdot \alpha R_{\alpha} h \cdot \chi^2 dx 
\leq - \alpha \int_{B} (\alpha R_{\alpha} h - h) h \chi^2 dx \nonumber \\
&= \lim_{n \rightarrow \infty} - \alpha \int_{B} (\alpha R_{\alpha} h_n - h_n) h \chi^2 dx \underset{\text{by \eqref{operatorequ}}}{=} \lim_{n \rightarrow \infty} - \int_{B} L^A (\alpha R_{\alpha} h_n) h \chi^2 dx \nonumber \\
&= \lim_{n \rightarrow \infty} \bigg( - \int_{B} L^A(\chi^2 \alpha R_{\alpha} h_n) h dx 
+ \int_{B} \langle 2A\nabla(\chi^2), \nabla \alpha R_{\alpha} h_n \rangle h dx 
+ \int_{B} L^A(\chi^2) \cdot \alpha R_{\alpha} h_n \cdot h\, dx \bigg) \nonumber \\
&\underset{\text{by \eqref{doublediwstr}}}{=\quad} \lim_{n \rightarrow \infty} \Big(
\int_{B} c\, \chi^2 \alpha R_{\alpha} h_n \cdot h\, dx 
+ \int_{B} \langle \mathbf{G}, \nabla (\chi^2 \alpha R_{\alpha} h_n) \rangle h\, dx 
- \int_{B} \tilde{f} \chi^2 \alpha R_{\alpha} h_n\, dx \nonumber \\
&\quad - \int_{B} \langle \tilde{\mathbf{F}}, \nabla (\chi^2 \alpha R_{\alpha} h_n) \rangle dx 
+ 2 \int_{B} \langle 2A \nabla \chi, \nabla (\alpha R_{\alpha} h_n) \rangle \chi h\, dx 
+ \int_{B} L^A(\chi^2) \alpha R_{\alpha} h_n \cdot h\, dx \Big) \nonumber \\
&= \lim_{n \rightarrow \infty} \bigg(
\int_{B} c\, \chi^2 \alpha R_{\alpha} h_n \cdot h\, dx 
+ \int_{B} \langle \mathbf{G}, \nabla (\chi \alpha R_{\alpha} h_n) \rangle \chi h\, dx 
- \int_{B} \tilde{f} \chi^2 \alpha R_{\alpha} h_n\, dx \nonumber \\
&\quad - \int_{B} \langle \chi \tilde{\mathbf{F}}, \nabla (\chi \alpha R_{\alpha} h_n) \rangle dx 
+ 2 \int_{B} \langle 2A\nabla \chi, \nabla (\chi \alpha R_{\alpha} h_n) \rangle h\, dx 
+ \int_{B} L^A(\chi^2) \alpha R_{\alpha} h_n \cdot h\, dx \nonumber \\
&\quad - \int_{B} \langle \tilde{\mathbf{F}}, \nabla \chi \rangle \chi \alpha R_{\alpha} h_n\, dx 
+ \int_{B} \langle \mathbf{G}, \nabla \chi \rangle \chi \alpha R_{\alpha} h_n \cdot h\, dx 
- 2 \int_{B} \langle 2A \nabla \chi, \nabla \chi \rangle \alpha R_{\alpha} h_n \cdot h\, dx \bigg) \nonumber \\
&\underset{\text{by \eqref{ralphfnconv}, \eqref{weaklyconh}}}{=} \int_{B} c\, \chi^2 \alpha R_{\alpha} h \cdot h\, dx 
+ \int_{B} \langle \mathbf{G}, \nabla (\chi \alpha R_{\alpha} h) \rangle \chi h\, dx 
- \int_{B} \tilde{f} \chi^2 \alpha R_{\alpha} h\, dx \nonumber \\
&\quad - \int_{B} \langle \chi \tilde{\mathbf{F}}, \nabla (\chi \alpha R_{\alpha} h) \rangle dx 
+ 2 \int_{B} \langle 2A \nabla \chi, \nabla (\chi \alpha R_{\alpha} h) \rangle h\, dx 
+ \int_{B} L^A(\chi^2) \alpha R_{\alpha} h \cdot h\, dx \nonumber \\
&\quad - \int_{B} \langle \tilde{\mathbf{F}}, \nabla \chi \rangle \chi \alpha R_{\alpha} h\, dx 
+ \int_{B} \langle \mathbf{G}, \nabla \chi \rangle \chi \alpha R_{\alpha} h \cdot h\, dx 
- 4 \int_{B} \langle A \nabla \chi, \nabla \chi \rangle \alpha R_{\alpha} h \cdot h\, dx \nonumber \\
&\leq \|\chi^2 c\|_{L^1(B)} \|h\|^2_{L^{\infty}(B)} 
+ \|\mathbf{G} \|_{L^2(B)} \|\chi h \|_{L^{\infty}(B)} \lambda^{-1/2} \mathcal{E}^0(\chi \alpha R_{\alpha} h, \chi \alpha R_{\alpha} h )^{1/2} \nonumber \\
&\quad + \|\chi^2 \tilde{f}\|_{L^1(B)} \|h\|_{L^{\infty}(B)}     
+ \Big( \lambda^{-1/2}\|\chi \tilde{\mathbf{F}} \|_{L^2(B)}  + 4 \mathcal{E}^0 (\chi, \chi)^{1/2} \|h\|_{L^{\infty}(B)} \Big) \mathcal{E}^0 (\chi \alpha R_{\alpha} h, \chi \alpha R_{\alpha} h  )^{1/2} \nonumber \\
&\quad + \|L^A (\chi^2) \|_{L^1(B)} \|h \|^2_{L^{\infty}(B)} 
+ \|\tilde{\mathbf{F}} \|_{L^1(B)} \| \nabla \chi \|_{L^{\infty}(B)} \| h \|_{L^{\infty}(B)} \nonumber \\
&\quad + \| \mathbf{G} \|_{L^1(B)} \|\nabla \chi \|_{L^{\infty}(B)} \| h \|_{L^{\infty}(B)} 
+ 4 \mathcal{E}^0 (\chi, \chi) \| h\|^2_{L^{\infty}(B)} \nonumber \\
&= C_3 \cdot \mathcal{E}^0 (\chi \alpha R_{\alpha} h, \chi \alpha R_{\alpha} h)^{1/2} + C_4, \label{criticineq}
\end{align}
where
\[
C_3 := \|\mathbf{G} \|_{L^2(B)} \|\chi h \|_{L^{\infty}(B)} \lambda^{-1/2} 
+ \lambda^{-1/2} \|\chi \tilde{\mathbf{F}} \|_{L^2(B)} 
+ 4 \mathcal{E}^0 (\chi, \chi)^{1/2} \|h\|_{L^{\infty}(B)},
\]
and
\begin{align*}
C_4 :=\; & \|\chi^2 c\|_{L^1(B)} \|h\|^2_{L^{\infty}(B)} 
+ \|\chi^2 \tilde{f}\|_{L^1(B)} \|h\|_{L^{\infty}(B)}
+ \|L^{A} (\chi^2)\|_{L^1(B)} \|h\|^2_{L^{\infty}(B)} \\
&+ \|\tilde{\mathbf{F}} \|_{L^1(B)} \|\nabla \chi \|_{L^{\infty}(B)} \|h\|_{L^{\infty}(B)} 
+ \| \mathbf{G} \|_{L^1(B)} \|\nabla \chi \|_{L^{\infty}(B)} \|h\|_{L^{\infty}(B)} 
+ 4 \mathcal{E}^0 (\chi, \chi) \|h\|^2_{L^{\infty}(B)}.
\end{align*}
Using \eqref{energybound}, \eqref{ralphfnconv}, \eqref{weaklyconh} and \eqref{criticineq}, it follows that
\begin{align*}
&\mathcal{E}^0(\chi \alpha R_{\alpha} h, \chi \alpha R_{\alpha} h ) 
\leq \liminf_{n \rightarrow \infty} \mathcal{E}^0(\chi \alpha R_{\alpha} h_n, \chi \alpha R_{\alpha} h_n ) \\
&\underset{\text{ by \eqref{energybound}}}{=} - \int_{B} \chi (L^{A} \chi)  \cdot(\alpha R_{\alpha} h)^2 dx 
- \alpha \int_{B} (\alpha R_{\alpha} h - h) \chi^2 \alpha R_{\alpha} h \, dx \\
&\quad - \int_{B} \langle 2A \nabla \chi, \nabla (\chi \alpha R_{\alpha} h) \rangle \chi \alpha R_{\alpha} h dx 
+ \int_{B} \langle 2A \nabla \chi, \nabla \chi \rangle (\alpha R_{\alpha} h)^2 dx \\
&\quad - \int_{B} \langle {\rm div}A, \nabla (\chi \alpha R_{\alpha} h) \rangle (\chi \alpha R_{\alpha} h) dx \\
&\underset{\text{by \eqref{criticineq}}}{\leq} \| \chi\|_{L^1(B)}  \| L^A \chi \|_{L^{\infty}(B)} \|h\|^2_{L^{\infty}(B)}+ C_3 \mathcal{E}^0 (\chi \alpha R_{\alpha} h, \chi \alpha R_{\alpha} h)^{1/2} + C_4 \\
&\quad + 2\| h \|_{L^{\infty}(B)} \mathcal{E}^0 (\chi, \chi)^{1/2} \mathcal{E}^0 (\chi \alpha R_{\alpha} h, \chi \alpha R_{\alpha} h)^{1/2} \\
&\quad + 2\| h \|^2_{L^{\infty}(B)} \mathcal{E}^0(\chi, \chi) 
+ \| {\rm div}A\|_{L^2(B)} \| h \|_{L^{\infty}(B)} \lambda^{-1/2} \mathcal{E}^0(\chi \alpha R_{\alpha} h, \chi \alpha R_{\alpha} h)^{1/2} \\
&\le C_5 \mathcal{E}^0 (\chi \alpha R_{\alpha}h, \chi \alpha R_{\alpha}h)^{1/2} + C_6,
\end{align*}
where 
$$
C_5 = C_3+2\| h \|_{L^{\infty}(B)} \mathcal{E}^0 (\chi, \chi)^{1/2}+\| {\rm div}A\|_{L^2(B)} \| h \|_{L^{\infty}(B)} \lambda^{-1/2}
$$ 
and 
$$
C_6 = \| \chi\|_{L^1(B)}  \| L^A \chi \|_{L^{\infty}(B)} \|h\|^2_{L^{\infty}(B)} + C_4 + 2\| h \|^2_{L^{\infty}(B)} \mathcal{E}^0(\chi, \chi).
$$
Thus, using Young's inequality,
\begin{equation} \label{uniformhestimr}
\mathcal{E}^0 (\chi \alpha R_{\alpha} h, \chi \alpha R_{\alpha} h) 
\leq C_5^2 + 2C_6 =: C_7,
\end{equation}
where $C_7$ is the constant independent of $\alpha$. Hence, by the Banach–Alaoglu theorem, there exist $\hat{h} \in D(\mathcal{E}^0)$ and an increasing sequence $(\alpha_k)_{k \geq 1} \subset [\lambda_0, \infty)$ with $\alpha_k \nearrow \infty$ as $k \to \infty$ such that
\begin{equation} \label{alphainfin}
\lim_{k \to \infty} \chi \alpha_k R_{\alpha_k} h = \hat{h} \quad \text{weakly in } D(\mathcal{E}^0).
\end{equation}
Let $L^{0} g ={\rm trace}(A \nabla^2 g) + \langle {\rm div}A, \nabla g \rangle$, $g \in C_0^{\infty}(B)$. Then, using Lemma \ref{coercievforml}(ii),
\begin{align} 
&-\int_{B}  \Big({\rm trace}(A \nabla^2 f) + \langle {\rm div}A, \nabla f \rangle \Big) \cdot g\,dx = \int_{B} \langle A \nabla f, \nabla g \rangle \,dx = \int_{B} \langle A \nabla g, \nabla f \rangle \,dx  \nonumber \\
&\quad= - \int_{B} L^{0}g \cdot f \,dx \label{intbypartn} \quad \text{ for any $f \in H^{2,p}(B) \cap H^{1,2}_0(B)$ and $g \in C_0^{\infty}(B)$}.
\end{align}
Now let $u \in C_0^{\infty}(B)$ be arbitrary. Then, using \eqref{ralphfnconv}, \eqref{weaklyconh}, \eqref{operatorequ} and \eqref{submarkoestim}, we have
\begin{align*}
&\left| \int_{B} (\hat{h} - \chi h) u \, dx \right| 
\underset{\text{by \eqref{alphainfin}}}{=} \lim_{k \to \infty} \left| \int_{B} \chi (\alpha_k R_{\alpha_k} h - h) u \, dx \right| \\
&\underset{\text{by \eqref{ralphfnconv}}}{=} \lim_{k \to \infty} \lim_{n \to \infty} \left| \int_{B} \chi (\alpha_k R_{\alpha_k} h_n - h_n) u \, dx \right| 
\underset{\text{by \eqref{operatorequ}}}{=} \lim_{k \to \infty} \lim_{n \to \infty} \left| \int_{B} \chi u (L^A R_{\alpha_k} h_n)  \, dx \right| \\
&\underset{\text{by \eqref{realizatio}}} {=}\lim_{k \to \infty} \lim_{n \to \infty} \left| \int_{B} \chi u \Big(
{\rm trace}(A \nabla^2 R_{\alpha_k} h_n) + \langle {\rm div}A, \nabla R_{\alpha_k} h_n \rangle \Big)  \, dx - \int_{B} \langle {\rm div}A, \chi \nabla R_{\alpha_k} h_n 	\rangle u\,dx  \right| 			\\
&\underset{\text{by \eqref{intbypartn}}} {=}\lim_{k \to \infty} \lim_{n \to \infty} \left| \int_{B} R_{\alpha_k} h_n   L^{0} (\chi u) - \int_{B}\langle {\rm div} A, \nabla (\chi R_{\alpha_k} h_n) \rangle u\,dx + \int_{B}\langle {\rm div} A, u\nabla \chi \rangle R_{\alpha_k} h_n   \, dx \right| \\
&\underset{\text{by \eqref{ralphfnconv}}}{=} \lim_{k \to \infty} \left| \int_{B} L^{0}(\chi u) \cdot R_{\alpha_k} h \, dx 
- \int_{B} \langle {\rm div} A, \nabla (\chi R_{\alpha_k} h) \rangle u \, dx 
+ \int_{B} \langle {\rm div} A, \nabla \chi \rangle R_{\alpha_k} h \cdot u \, dx \right| \\
&\underset{\text{by }\eqref{submarkoestim}}{\leq} \lim_{k \to \infty} 
\frac{1}{\alpha_k} \| L^{0}(\chi u) \|_{L^1(B)} \| h \|_{L^{\infty}(B)} 
+ \frac{1}{\alpha_k}\| {\rm div} A \|_{L^2(B)} \lambda^{-1/2} \mathcal{E}^0(\chi \alpha_k R_{\alpha_k} h, \chi \alpha_k R_{\alpha_k} h)^{1/2}  \|u\|_{L^{\infty}(B)} \\
&\qquad  +\frac{1}{\alpha_k} \| {\rm div} A \|_{L^1(B)} \| \chi \|_{L^{\infty}(B)} \| h \|_{L^{\infty}(B)} \|u\|_{L^{\infty}(B)}  \\
&\underset{\text{by \eqref{uniformhestimr}}}{\leq} \lim_{k \to \infty} \frac{1}{\alpha_k} 
\Big( \| L^{0}(\chi u) \|_{L^1(B)} \| h \|_{L^{\infty}(B)} 
+ \| {\rm div} A \|_{L^2(B)} \lambda^{-1/2} C_7 \|u\|_{L^{\infty}(B)}
\\
&\qquad \qquad + \| {\rm div} A \|_{L^1(B)} \| \chi \|_{L^{\infty}(B)} \| h \|_{L^{\infty}(B)} \|u\|_{L^{\infty}(B)}\Big) = 0.
\end{align*}
Therefore, we conclude that $\chi h = \hat{h} \in D(\mathcal{E}^0)=H^{1,2}_0(B)$, and hence $\chi h \in H^{1,2}_0(B)$, as desired.
\end{proof}
\centerline{}
\centerline{}
Now, we are ready to present the proof of the main result, Theorem \ref{maintheore} in introduction. \\ \\
{\bf Proof of Theorem \ref{maintheore}} \;
Fix an arbitrary point $x_0 \in U$. Since $U$ is open, there exists $R_0 > 0$ such that  $\overline{B}_{R_0}(x_0) \subset U$. Let $R \in (0, R_0)$ be arbitrary. To prove the desired regularity of $h$, it suffices to show that $h \in H^{1,2}_{loc}(B_R(x_0))$. Choose a smooth cutoff function $\chi \in C^\infty_0(B_{R_0}(x_0))$ satisfying $\chi \equiv 1$ on $B_R(x_0)$ and $0 \leq \chi \leq 1$ on $\mathbb{R}^d$. Since each coefficient $a_{ij}$ belongs to $VMO(B_{R_0}(x_0))$ for $1 \leq i,j \leq d$, we apply \cite[Lemma 5.2]{L25bv} to obtain functions $\hat{a}_{ij} \in VMO_{\omega_1} \cap L^\infty(\mathbb{R}^d)$ for some modulus $\omega_1$ such that $\hat{a}_{ij} = a_{ij}$ on $B_{R_0}(x_0)$. Let us denote $\hat{A} := (\hat{a}_{ij})_{1 \leq i,j \leq d}$. We now define a modified matrix of functions $\overline{A} := (\overline{a}_{ij})_{1 \leq i,j \leq d}$ via
\[
\overline{A}(x) := \chi(x)\big( \hat{A}(x) - \lambda_{B_{R_0}(x_0)} \, id \big) + \lambda_{B_{R_0}(x_0)} \, id, \quad x \in \mathbb{R}^d,
\]
where $id$ is the $d \times d$ identity matrix. Note that $\chi \in VMO_{\omega_2} \cap L^\infty(\mathbb{R}^d)$ for some modulus $\omega_2$, so by setting $\omega := \omega_1 + \omega_2$, it follows from \cite[Proposition 5.1]{L25bv} that for each $1 \leq i,j \leq d$ $\overline{a}_{ij} \in VMO_{\omega} \cap L^\infty(\mathbb{R}^d)$. Observe that for each $x \in B_R(x_0)$, we have
\[
\overline{a}_{ij}(x) = \hat{a}_{ij}(x) = a_{ij}(x), \quad \text{for all } 1 \leq i,j \leq d.
\]
Furthermore, the matrix $\overline{A}(x)$ satisfies the following:  for a.e. $x \in \mathbb{R}^d$ and all $\xi \in \mathbb{R}^d$,
\[
\lambda_{B_{R_0}(x_0)} |\xi|^2 \leq \langle \overline{A}(x) \xi, \xi \rangle  \quad \text{and} \quad \max_{i,j} |\overline{a}_{ij}(x)| \leq M_{B_{R_0}(x_0)} + 2 \lambda_{B_{R_0}(x_0)}. %\quad \text{ for a.e. $x \in \mathbb{R}^d$ and all $\xi \in \mathbb{R}^d$. }
\]
We observe that $\operatorname{div} \overline{A} = \operatorname{div} A \in L^2(B_R(x_0), \mathbb{R}^d)$. Also, $h$ satisfies the following identity for every test function $\varphi \in C_0^\infty(B_R(x_0))$:
\begin{equation} \label{eq:localform}
\int_{B_R(x_0)} \left( \operatorname{trace}(\overline{A} \nabla^2 \varphi) + \langle \mathbf{G}, \nabla \varphi \rangle + c \varphi \right) h \, dx = \int_{B_R(x_0)} \tilde{f} \varphi \, dx + \int_{B_R(x_0)} \langle \tilde{\mathbf{F}}, \nabla \varphi \rangle \, dx.
\end{equation}
The matrix  $\overline{A}$ satisfies the structural assumptions in \textbf{(Hypo)}. Therefore, we apply Theorem~\ref{submaintheor} to \eqref{eq:localform} and conclude that $h \in H^{1,2}_{loc}(B_R(x_0))$, as desired. \\
\text{}\hfill $\square$

\section{Application to the uniqueness results} \label{sec4}
The local regularity result established in the main theorem enables us to convert solutions of the stationary Fokker–Planck equations into weak solutions of a divergence form equation. Based on this, our main theorem can be applied to derive a uniqueness result for solutions to the stationary Fokker–Planck equations.
\begin{theo} \label{applicationth}
Let $A = (a_{ij})_{1 \leq i,j \leq d}$ be a symmetric matrix of functions whose components satisfy $a_{ij} \in VMO_{loc}(\mathbb{R}^d)$ for all $1 \leq i,j \leq d$. Assume that the condition \eqref{ellipticit} holds for some strictly positive constants $\lambda$ and $M$. Let $C = (c_{ij})_{1 \leq i,j \leq d}$ be an anti-symmetric matrix of functions satisfying $c_{ij} = -c_{ji} \in L^{\infty}(\mathbb{R}^d)$ for all $1 \leq i,j \leq d$. Assume that ${\rm div}(A+C) \in L^2_{loc}(\mathbb{R}^d, \mathbb{R}^d)$. Let $\rho \in H^{1,2}_{loc}(\mathbb{R}^d) \cap L^{\infty}(\mathbb{R}^d)$ be such that for some constant $\theta>0$, $\rho(x) > \theta$ for a.e. $x \in \mathbb{R}^d$. Then, the following statements hold:
\begin{itemize}
\item[(i)]
\begin{align*}
%&\int_{\mathbb{R}^d} \frac{1}{\rho} {\rm div} \Big( \rho(A+C) \nabla f \Big)\, \rho \, dx \\
 \int_{\mathbb{R}^d} \bigg( {\rm trace}(A \nabla^2 f) + \left\langle {\rm div}(A+C) + \frac{1}{\rho}(A+C^T) \nabla \rho, \nabla f \right\rangle \bigg)\, \rho\, dx=0 \quad \text{ for all $f \in C_0^{\infty}(\mathbb{R}^d)$.}
\end{align*}
\item[(ii)]
If $h \in L^{\infty}_{loc}(\mathbb{R}^d)$ with $h \geq 0$ in $\mathbb{R}^d$ satisfies
\[
\int_{\mathbb{R}^d} \bigg( {\rm trace}(A \nabla^2 f) + \left\langle {\rm div}(A+C) + \frac{1}{\rho}(A+C^T) \nabla \rho, \nabla f \right\rangle \bigg)\, h\, dx = 0 \quad \text{for all } f \in C_0^{\infty}(\mathbb{R}^d),
\]
then there exists a constant $c \geq 0$ such that $h = c \rho$ in $\mathbb{R}^d$.
\item[(iii)]
If $h \in L^{\infty}(\mathbb{R}^d)$ satisfies
\[
\int_{\mathbb{R}^d} \bigg( {\rm trace}(A \nabla^2 f) + \left\langle {\rm div}(A+C) + \frac{1}{\rho}(A+C^T) \nabla \rho, \nabla f \right\rangle \bigg)\, h\, dx = 0 \quad \text{for all } f \in C_0^{\infty}(\mathbb{R}^d),
\]
then there exists a constant $c \in \mathbb{R}$ such that $h = c \rho$ in $\mathbb{R}^d$.
\end{itemize}
\end{theo}
\begin{proof}
(i) Let $f \in C_0^{\infty}(\mathbb{R}^d)$.
Then, by \cite[Proposition 5.7]{L25bv}, 
\begin{equation} \label{divtracrep}
 \frac{1}{\rho} {\rm div} \Big( \rho(A+C) \nabla f \Big) = {\rm trace}(A \nabla^2 f) + \left\langle {\rm div}(A+C) + \frac{1}{\rho}(A+C^T) \nabla \rho, \nabla f \right\rangle.
\end{equation}
Take an open ball $V$ so that ${\rm supp}(f) \subset V$. Let $\chi \in C_0^{\infty}(\mathbb{R}^d)$ with $\chi \equiv 1$ on $\overline{V}$. Then,
\begin{align*}
&\int_{\mathbb{R}^d} \bigg( {\rm trace}(A \nabla^2 f) + \left\langle {\rm div}(A+C) + \frac{1}{\rho}(A+C^T) \nabla \rho, \nabla f \right\rangle \bigg)\, \rho\, dx \\
&=\int_{\mathbb{R}^d}  {\rm div} \Big( \rho(A+C) \nabla f \Big)\,dx =  \int_{\mathbb{R}^d}  {\rm div} \Big( \rho(A+C) \nabla f \Big) \chi \,dx \\
& = -\int_{V} \langle \rho(A+C) \nabla f, \nabla \chi \rangle\,dx=0.
\end{align*}
(ii)  First, by Theorem \ref{maintheore}, $h \in H^{1,2}_{loc}(\mathbb{R}^d)$.
Let $\hat{h}:=\frac{h}{\rho}$. Then, $\hat{h} \in H^{1,2}_{loc}(\mathbb{R}^d)$ with $\hat{h} \geq 0$ in $\mathbb{R}^d$ and it follows from \eqref{divtracrep} that
\begin{align*}
0=&\int_{\mathbb{R}^d} \bigg( {\rm trace}(A \nabla^2 f) + \left\langle {\rm div}(A+C) + \frac{1}{\rho}(A+C^T) \nabla \rho, \nabla f \right\rangle \bigg)\, h\, dx  \\
& = \int_{\mathbb{R}^d} {\rm div} \Big( \rho(A+C) \nabla f \Big) \cdot \hat{h}\,dx = -\int_{\mathbb{R}^d} \langle \rho (A+C^T) \nabla \hat{h}, \nabla f \rangle \,dx \quad \text{ for all $f \in C_0^{\infty}(\mathbb{R}^d)$}.
\end{align*}
By the Liouville-type theorem \cite[Theorem 9.11(i)]{F07}, $\hat{h}$ is constant, so that the assertion follows. \\ \\
(iii) As in (ii), $h \in H^{1,2}_{loc}(\mathbb{R}^d)$. Moreover, $\hat{h}:=\frac{h}{\rho} \in H^{1,2}_{loc}(\mathbb{R}^d) \cap L^{\infty}(\mathbb{R}^d)$ satisfies
$$
 \int_{\mathbb{R}^d} \langle \rho (A+C^T) \nabla \hat{h}, \nabla f \rangle \,dx=0 \quad \text{ for all $f \in C_0^{\infty}(\mathbb{R}^d)$}.
$$
Using the Liouville-type theorem \cite[Theorem 9.11(ii)]{F07}, $\hat{h}$ is constant, so that the assertion follows.
\end{proof}
\centerline{}
\noindent
To present a concrete example, we consider the following function:
Let $\varepsilon \in (0, 10^{-2026})$ be fixed. Define
\begin{equation} \label{defnpsi}
\Psi(t) := 
\begin{cases}
1, & t \in (-\infty, -1), \\[6pt]
1 + \displaystyle\int_{-1}^{t} |s|^{-\frac{1}{2+\varepsilon}}\,ds, &  t \in [-1,1], \\[6pt]
1 + \displaystyle\int_{-1}^{1} |s|^{-\frac{1}{2+\varepsilon}}\,ds, & t \in (1, \infty).
\end{cases}
\end{equation}
Then $\Psi \in C(\mathbb{R}) \cap L^{\infty}(\mathbb{R})$, and its derivative $\Psi' \in \bigcap_{p \in (1,\; 2+\varepsilon)} L^p(\mathbb{R})$ is given by
\[
\Psi'(t) := 
\begin{cases}
0, & t \in (-\infty, -1), \\[6pt]
|t|^{-\frac{1}{2+\varepsilon}}, &  t \in (-1,0) \cup (0,1), \\[6pt]
0, & t \in (1, \infty).
\end{cases}
\]
In particular, $\Psi' \notin L_{loc}^{2+\varepsilon}(\mathbb{R})$, and
\[
1 \leq \Psi(t) \leq 1 + \int_{-1}^{1} |s|^{-\frac{1}{2+\varepsilon}}\,ds \quad \text{for all } t \in \mathbb{R}.
\]

\begin{exam} \label{ourexam1}
Let $d \geq 2$ and $\varepsilon \in (0,10^{-2026})$ be fixed. Let $\Psi$ be the function defined in \eqref{defnpsi}. Define
$$
\rho(x_1, \ldots, x_d) = \prod_{i=1}^d \Psi(x_i), \quad x =(x_1, \ldots, x_d) \in \mathbb{R}^d.
$$
Then $\rho \in C(\mathbb{R}^d) \cap L^{\infty}(\mathbb{R}^d) \cap H^{1,2}_{loc}(\mathbb{R}^d)$, and $\nabla \rho \notin L^{2+\varepsilon}_{loc}(\mathbb{R}^d)$. Moreover,
$$
1 \leq \rho(x) \leq \left( 1 + \int_{-1}^{1} |s|^{-\frac{1}{2+\varepsilon}}\,ds \right)^d \quad \text{for all } x \in \mathbb{R}^d.
$$
Let $A = I$, $\mathbf{G} = \frac{1}{\rho} \nabla \rho$. Then
\begin{equation} \label{infinitesinvar}
\int_{\mathbb{R}^d} \left( {\rm trace}(A \nabla^2 f)  + \langle \mathbf{G}, \nabla f \rangle \right) \rho\,dx = 0 \quad \text{for all } f \in C_0^{\infty}(\mathbb{R}^d).
\end{equation}
Assume that $h \in L^{\infty}_{loc}(\mathbb{R}^d)$ with $h \geq 0$ in $\mathbb{R}^d$ satisfies
\begin{equation} \label{infininvah}
\int_{\mathbb{R}^d} \left({\rm trace}(A \nabla^2 f)  + \langle \mathbf{G}, \nabla f \rangle \right) h \,dx = 0 \quad \text{for all } f \in C_0^{\infty}(\mathbb{R}^d).
\end{equation}
Then it follows from Theorem \ref{applicationth}(ii) that $h$ is a constant multiple of $\rho$. If $h \in L^{\infty}(\mathbb{R}^d)$ satisfies \eqref{infininvah}, then it follows from Theorem \ref{applicationth}(iii) that $h$ is also a constant multiple of $\rho$.
\end{exam}

\begin{exam}
Let $d \geq 3$, and let $\rho$ be the function defined in Example \ref{ourexam1}. Define $A = \rho\, id$, and let $C = (c_{ij})_{1 \leq i,j \leq d}$ be an anti-symmetric matrix of functions given by
\[
c_{ij} =
\begin{cases}
\rho, & \text{if } (i,j) = (1,d), \\[6pt]
-\rho, & \text{if } (i,j) = (d,1), \\[6pt]
0, & \text{otherwise}.
\end{cases}
\]
Let $\mathbf{G} = {\rm div}(A + C) = \nabla \rho + (-\partial_d \rho, 0, \ldots, \partial_1 \rho)$. Then, $\mathbf{G} \in L^2_{loc}(\mathbb{R}^d, \mathbb{R}^d)$, and $\mathbf{G} \notin L^{2+\varepsilon}_{loc}(\mathbb{R}^d, \mathbb{R}^d)$. Moreover, \eqref{infinitesinvar} holds.
Now if $h \in L^{\infty}_{loc}(\mathbb{R}^d)$ with $h \geq 0$ in $\mathbb{R}^d$ satisfies \eqref{infininvah}, then by Theorem \ref{applicationth}(ii), $h$ is a constant. If $h \in L^{\infty}(\mathbb{R}^d)$ satisfies \eqref{infininvah}, then it follows from Theorem \ref{applicationth}(iii) that $h$ is also a constant.
\end{exam}

%\subsection*{Acknowledgments}
%This research was supported by the regional innovation system \& education (RISE)-(Regional Growth Innovation LAB) program through the Gyeongbuk RISE Center, funded by the Ministry of Education (MOE) and the Gyeongsangbuk-do, Republic of Korea (2025-rise-15-105).

%\subsection*{Author contributions}
%This is a single-author manuscript, entirely written by the author.

%\subsection*{Conflict of interest}
%The author declares no conflict of interest.

%\subsection*{Ethics approval}
%The conducted research does not involve human participants or animal subjects.

%\subsection*{Data availability statement}
%Not applicable.

\centerline{}
\centerline{}
Haesung Lee\\
Department of Mathematics and Big Data Science,  \\
Kumoh National Institute of Technology, \\
Gumi, Gyeongsangbuk-do 39177, Republic of Korea, \\
E-mail: fthslt@kumoh.ac.kr, \; fthslt14@gmail.com
\end{document}